\documentclass[bjps, preprint]{imsart}

\usepackage{amsmath, amsthm, amstext, amssymb, amsfonts, graphicx, bm, subfigure}

\allowdisplaybreaks

\RequirePackage[OT1]{fontenc}
\RequirePackage[numbers]{natbib}
\RequirePackage[colorlinks,citecolor=blue,urlcolor=blue]{hyperref}

\numberwithin{equation}{section}

\newcommand{\PR}{\mathbb P}
\newcommand{\G}{\mathcal G}

\theoremstyle{plain}

\newtheorem{theorem}{Theorem}[section]

\newtheorem{remark}{Remark}

\begin{document}

\begin{frontmatter}

\title{Asymptotics for sparse exponential random graph models}
\runtitle{Asymptotics for sparse ERGMs}

\begin{aug}
\author{\fnms{Mei}
\snm{Yin}\thanksref{t1}\ead[label=e1]{mei.yin@du.edu}}
\and
\author{\fnms{Lingjiong}
\snm{Zhu}\thanksref{t2}\ead[label=e2]{ling@cims.nyu.edu}}

\runauthor{Mei Yin and Lingjiong Zhu}

\affiliation{University of Denver\thanksmark{t1} and University of Minnesota\thanksmark{t2}}

\address{Department of Mathematics\\
University of Denver\\
\printead{e1}\\
\phantom{E-mail:\ }}

\address{School of Mathematics\\
University of Minnesota\\
\printead{e2}\\
\phantom{E-mail:\ }}
\end{aug}

\begin{abstract}
We study the asymptotics for sparse exponential random graph models where the parameters may depend on the number
of vertices of the graph. We obtain exact estimates for the mean and variance of the limiting probability distribution and the limiting log partition function of the edge-(single)-star model. They are in sharp contrast to the corresponding asymptotics in dense exponential random graph models. Similar analysis is done for directed sparse exponential random graph models parametrized by edges and multiple outward stars.
\end{abstract}

\begin{keyword}[class=AMS]
\kwd[Primary ]{05C80} \kwd[; secondary ]{60C05} \kwd{05C35}
\end{keyword}

\begin{keyword}
\kwd{sparse random graphs} \kwd{exponential random graphs} \kwd{asymptotics}
\end{keyword}

\end{frontmatter}

\section{Introduction}
\label{intro}
Exponential random graphs are a class of graph ensembles of fixed vertex number $n$
defined by analogy with the Boltzmann ensemble of statistical mechanics.
Let $\{\epsilon_p\}$ be a set of local features of a single graph,
for example the number of edges or copies of any finite subgraph,
as well as more complicated characteristics including the degree sequence or degree distribution, and combinations thereof.
These quantities play a role similar to energy in statistical mechanics.
Let $\{\beta_p\}$ be a set of inverse temperature parameters whose values we are free to choose.
By varying these parameters, one could analyze the influence of different local features on the global structure of the graph.
Let $\G_n$ be the set of all possible graphs (undirected and with no self-loops or multiple edges in the simplest case) on $n$ vertices.
The $k$-parameter family of exponential random graphs is defined by assigning a probability $\PR^{(n)}(G_n)$ to every graph $G_n$ in $\G_n$:
\begin{equation}
\PR^{(n)}(G_n)=Z_n(\beta_{1}^{(n)},\beta_{2}^{(n)},\ldots,\beta_{k}^{(n)})^{-1}\exp\left[\sum_{p=1}^k \beta_p^{(n)}\epsilon_p(G_n)\right],
\end{equation}
where $Z_n$ is the partition function,
\begin{equation}
Z_n(\beta_{1}^{(n)},\beta_{2}^{(n)},\ldots,\beta_{k}^{(n)})=\sum_{G_n\in \G_n}\exp\left[\sum_{p=1}^k \beta_p^{(n)}\epsilon_p(G_n)\right].
\end{equation}

These rather general models are widely used to model real-world networks, such as the Internet,
the World Wide Web, social networks, and biological networks, as they are able to capture a wide variety of common network
tendencies by representing a complex global structure through a set of tractable local features,
see e.g. Newman \cite{Newman} and Wasserman and Faust \cite{Wasserman}.
They are particularly useful when one wants to construct models that resemble observed networks as closely as possible
but without going into details of the specific process underlying network formation. Since real-world networks are often very large in size, a pressing objective is to understand the asymptotics of the mean and variance of the limiting probability distribution $\PR^{(n)}$ and the limiting log partition function $\log Z_n$.
By differentiating $\log Z_n$ with respect to appropriate parameters $\beta_p$, averages of various quantities of interest may be derived.
In particular, a phase transition occurs when $\log Z_n$ is non-analytic, since it is the generating function for the limiting expectations of other random variables.
Computation of $\log Z_n$ is also important in statistics because it is crucial
for carrying out maximum likelihood estimates and Bayesian inference of unknown parameters.

Exponential models described above have been extensively studied over the last decades. We refer to Besag \cite{Besag},
Snijders et al. \cite{Snijders}, Rinaldo et al. \cite{Rinaldo}, van der Hofstad \cite{Hofstad}, and Fienberg \cite{FienbergI, FienbergII} for history and a review of developments.
The past few years especially has witnessed (exponentially) growing
attention in exponential models and their variations. Many investigations have been centered on dense graphs (number of edges comparable to the square of number of vertices), and emphasis has been made on the variational principle of the limiting free energy, concentration of the limiting probability distribution,
phase transitions and asymptotic structures. See e.g. Chatterjee and Varadhan \cite{ChatterjeeVaradhan}, Chatterjee and Diaconis \cite{ChatterjeeDiaconis},
Radin and Yin \cite{Radin}, Lubetzky and Zhao \cite{LubetzkyZhao}, Radin and Sadun \cite{RadinII, RadinIV}, Radin et al. \cite{RadinIII},
Kenyon et al. \cite{Kenyon}, Yin \cite{Yin}, Yin et al. \cite{YinII}, Kenyon and Yin \cite{KenyonYin}, Aristoff and Zhu \cite{AristoffZhu, AristoffZhuII}, and Zhu \cite{Zhu}. Most of these papers utilize the theory of graph limits as developed by Lov\'{a}sz and coauthors (V. T. S\'{o}s, B. Szegedy, C. Borgs, J. Chayes, K. Vesztergombi, etc.), who have constructed this unified and elegant theory in a sequence of
papers \cite{BCLSV1, BCLSV2, BCLSV3, Lovasz2009, LS}. See the book by Lov\'{a}sz
\cite{Lov} for a comprehensive account and references.

Although the graph limit theory is tailored to dense graphs, parallel theories for sparse graphs are likewise emerging. See e.g. Benjamini and Schramm \cite{BS} and Aldous and Lyons \cite{AL} where the notion of local weak convergence is discussed, which is relevant to the sparse graph setting since the vast majority of sparse graphs are locally tree-like.
See also the new works of Borgs et al. \cite{BCCZ1, BCCZ2} that are making progress towards
generalizing the existing $L^\infty$ theory of dense graph limits by developing a limiting object for sparse graph sequences based on $L^p$ graphons. Their interest in sparse graphs is well justified, as most networks data are sparse in the real world. Biomedical signals tend to have sparse depictions when expressed in a proper basis; A gene network is sparse since a regulatory pathway involves only a small number of genes; The neural representation of sounds
in the auditory cortex of unanesthetized animals is sparse since the fraction of neurons active at a given
instant is small. More examples may be found in Golub et al. \cite{Golub}, Guyon et al. \cite{Guyon}, Hrom\'{a}dka et al. \cite{Hromadka},
and Ye and Liu \cite{YeLiu}. The present investigation will focus on sparse exponential random graph models.
They are indeed harder to study than their dense counterparts, nevertheless, some important discoveries have recently been
made by Chatterjee and Dembo \cite{ChatterjeeDembo}.

The rest of this paper is organized as follows. In Section \ref{undirected} we analyze the asymptotic features of the undirected exponential model parametrized by various subgraph densities and derive exact estimates for the mean and variance of the limiting probability distribution (Theorems \ref{meanundir} and \ref{varianceundir}) and the limiting log partition function of the edge-(single)-star model (Theorem \ref{sparse}) under a sparsity assumption about the parameters. In Section \ref{directed} we analyze the asymptotic features of the
directed exponential model parametrized by edges and multiple outward stars and derive exact estimates for the mean and variance of the limiting probability distribution (Theorems \ref{meandir}, \ref{variancedir}, \ref{caselambda} and \ref{infdir}) and the limiting log partition function (Remark \ref{remark}) under different sparsity assumptions about the parameters.

\section{Undirected Graphs}
\label{undirected}
Consider undirected graphs $G_n$ on $n$ vertices, where a graph is represented by a matrix
$X_n=(X_{ij})_{1\le i<j \le n}$ with each $X_{ij} \in \{0,1\}$.
Here, $X_{ij} = 1$ means there is an edge between vertex $i$ and
vertex $j$; otherwise, $X_{ij}=0$. Give the set of such graphs the probability
\begin{equation}\label{ll}
{\mathbb P}^{(n)}(G_n) = Z_n(\beta_{1}^{(n)},\beta_{2}^{(n)},\ldots,\beta_{k}^{(n)})^{-1}\exp\left[n^2\left(\sum_{p=1}^{k}\beta_{p}^{(n)} t(H_p, G_n)\right)\right],
\end{equation}
where $H_p$ is a finite simple graph with vertex set $V(H_p)=\{1,\ldots,v(H_p)\}$ and edge set $E(H_p)$, $t(H_p, G_n)$ is the homomorphism density of $H_p$ in $G_n$, i.e., the probability that a random vertex map $V(H_p) \rightarrow V(G_n)$ is edge-preserving,
\begin{equation}
t(H_{p},G_{n})=\frac{|\mbox{hom}(H_{p},G_{n})|}{|V(G_n)|^{|V(H_{p})|}},
\end{equation}
and $Z_n(\beta_{1}^{(n)},\beta_{2}^{(n)},\ldots,\beta_{k}^{(n)})$ is the appropriate normalization. The parameters $\beta_p^{(n)}$ are scaled according to the number
of vertices of the graph,
\begin{equation}
\label{scaling1}
\beta_{p}^{(n)}=\beta_{p}\alpha_{n},\quad p=1,2,\ldots,k,
\end{equation}
where $\alpha_{n}\rightarrow\infty$ as $n\rightarrow\infty$. We are interested in the situation where a typical random graph sampled from the exponential model $\PR^{(n)}$ (\ref{ll}) is sparse, i.e., the probability that there is an edge between vertex $i$ and vertex $j$ goes to $0$ as $n\rightarrow\infty$. One natural question to ask is when this indeed happens, what is the speed of the graph towards sparsity? This in general is a rather difficult question, nevertheless when the parameters $\beta_i$ are all negative, we are able to provide some concrete answers. The following Theorems \ref{meanundir} and \ref{varianceundir} give the mean and variance of the limiting probability distribution ${\mathbb P}^{(n)}$ under a sparsity assumption about the parameters.

\begin{theorem}
\label{meanundir}
Assume that $\beta_{1},\ldots,\beta_{k}$ are all negative and $H_{1}$ denotes a single edge. Let us further
assume that $\lim_{n\rightarrow\infty}n^{2}e^{2\alpha_{n}\beta_{1}}=0$ and \\ $\lim_{n\rightarrow\infty}\frac{\alpha_n}{n}=0$. Then
\begin{equation}
\lim_{n\rightarrow\infty}\frac{\mathbb{P}^{(n)}(X_{1i}=1)}{e^{2\alpha_{n}\beta_1}}=1, \hspace{0.2cm} i\neq 1.
\end{equation}
\end{theorem}

\begin{proof}
By symmetry, the fact that $X_{1i}\in\{0,1\}$ and the definition of the probability measure $\mathbb{P}^{(n)}$ \eqref{ll},
\begin{align}
\mathbb{P}^{(n)}(X_{1i}=1)
&=\mathbb{E}^{(n)}[X_{1i}]
\\
&=\frac{1}{n-1}\mathbb{E}^{(n)}\left[\sum_{i=2}^nX_{1i}\right]
\nonumber
\\
&=\frac{1}{n-1}\frac{2^{\binom{n}{2}}\mathbb{E}[\sum_{i=2}^nX_{1i}e^{\alpha_{n}n^{2}\sum_{p=1}^{k}\beta_{p}t(H_{p},G_n)}]}
{2^{\binom{n}{2}}\mathbb{E}[e^{\alpha_{n}n^{2}\sum_{p=1}^{k}\beta_{p}t(H_{p},G_n)}]},
\nonumber
\end{align}
where $\mathbb{E}$ is the expectation associated with the uniform measure, i.e., each possible graph configuration is weighted equally at $1/2^{\binom{n}{2}}$.

\noindent First, let us analyze the denominator. On one hand, taking $G_n$ to be the empty graph gives $t(H_p, G_n)=0$ for any simple graph $H_p$, which implies that
\begin{equation}
2^{\binom{n}{2}}\mathbb{E}[e^{\alpha_{n}n^{2}\sum_{p=1}^{k}\beta_{p}t(H_{p},G_n)}]\geq 2^{\binom{n}{2}}\frac{1}{2^{\binom{n}{2}}}=1.
\end{equation}
On the other hand, since $\beta_{1},\ldots,\beta_{k}$ are all negative and $t(H_1, G_n)$ measures the edge density of $G_n$, using $(X_{ij})_{1\le i<j \le n}$ are iid Bernoulli under the uniform measure, we have
\begin{align}
2^{\binom{n}{2}}\mathbb{E}[e^{\alpha_{n}n^{2}\sum_{p=1}^{k}\beta_{p}t(H_{p},G_n)}]
&\leq 2^{\binom{n}{2}}\mathbb{E}[e^{\alpha_{n}n^{2}\beta_{1}t(H_{1},G_n)}]
\\
&=2^{\binom{n}{2}}\mathbb{E}[e^{2\alpha_{n}\beta_{1}\sum_{1\leq i<j\leq n}X_{ij}}]
\nonumber
\\
&=2^{\binom{n}{2}}\left(\frac{1+e^{2\alpha_{n}\beta_{1}}}{2}\right)^{\binom{n}{2}}
\nonumber
\\
&\rightarrow 1.
\nonumber
\end{align}
Next, let us analyze the numerator. On one hand, since $t(H_1, G_n)$ measures the edge density of $G_n$, a graph $G_n$ with only one edge $X_{1i}=1$ for some $2\leq i\leq n$ gives $t(H_1, G_n)=2/n^2$ and carries a weight of $1/2^{\binom{n}{2}}$, which implies that
\begin{align}
&2^{\binom{n}{2}}\mathbb{E}\left[\sum_{i=2}^nX_{1i}e^{\alpha_{n}n^{2}\sum_{p=1}^{k}\beta_{p}t(H_{p},G_n)}\right]
\\
&\geq
2^{\binom{n}{2}}\sum_{i=2}^n\mathbb{E}\left[e^{\alpha_{n}n^{2}\sum_{p=1}^{k}\beta_{p}t(H_{p},G_n)}\bigg|X_{1i}=1,X_{i'j'}=0,(i',j')\neq(1,i)\right]
\nonumber
\\
&\qquad\qquad\qquad\cdot
\mathbb{P}\left(X_{1i}=1,X_{i'j'}=0,(i',j')\neq(1,i)\right)
\nonumber
\\
&= e^{2\alpha_{n}\beta_{1}}\sum_{i=2}^n
\mathbb{E}\left[e^{\alpha_{n}n^{2}\sum_{p=2}^{k}\beta_{p}t(H_{p},G_n)}\bigg|X_{1i}=1,X_{i'j'}=0,(i',j')\neq(1,i)\right]
\nonumber
\\
&=\sum_{i=2}^n e^{2\alpha_{n}\beta_{1}+\alpha_n n^2 \sum_{p=2}^k \beta_p c_p n^{-v(H_p)}}
\nonumber
\\
&=(n-1)e^{2\alpha_{n}\beta_{1}+\alpha_n n^2 \sum_{p=2}^k \beta_p c_p n^{-v(H_p)}},
\nonumber
\end{align}
where $v(H_p)\geq 3$ denotes the number of vertices of $H_p$ and $c_p$ is the number of homomorphisms from $H_p$ into $G_n$ (a graph with only one edge) and so is a constant that only depends on $H_p$.

\noindent On the other hand, since $\beta_{1},\ldots,\beta_{k}$ are all negative and $t(H_1, G_n)$ measures the edge density of $G_n$, using $(X_{ij})_{1\le i<j \le n}$ are iid Bernoulli under the uniform measure, we have
\begin{align}
&2^{\binom{n}{2}}\mathbb{E}\left[\sum_{i=2}^nX_{1i}e^{\alpha_{n}n^{2}\sum_{p=1}^{k}\beta_{p}t(H_{p},G_n)}\right]
\leq 2^{\binom{n}{2}}\mathbb{E}\left[\sum_{i=2}^nX_{1i}e^{\alpha_{n}n^{2}\beta_{1}t(H_{1},G_n)}\right]
\\
&=2^{\binom{n}{2}}\mathbb{E}\left[\sum_{i=2}^nX_{1i}e^{2\alpha_{n}\beta_{1}\sum_{1\leq i<j\leq n}X_{ij}}\right]
\nonumber
\\
&=2^{\binom{n}{2}}\frac{1}{2\alpha_{n}}\frac{\partial}{\partial\beta_{1}}\mathbb{E}\left[e^{2\alpha_{n}\beta_{1}\sum_{i=2}^nX_{1i}}\right]\mathbb{E}\left[e^{2\alpha_{n}\beta_{1}\sum_{2\leq i<j\leq n}X_{ij}}\right]
\nonumber
\\
&=2^{\binom{n}{2}}\frac{1}{2\alpha_{n}}\frac{\partial}{\partial\beta_{1}}\left(\frac{1+e^{2\alpha_{n}\beta_{1}}}{2}\right)^{n-1}\left(\frac{1+e^{2\alpha_{n}\beta_{1}}}{2}\right)^{\binom{n-1}{2}}
\nonumber
\\
&=(n-1)e^{2\alpha_{n}\beta_{1}}\left(1+e^{2\alpha_{n}\beta_{1}}\right)^{\binom{n}{2}-1}.
\nonumber
\end{align}
The conclusion follows when we apply the scaling assumption.
\end{proof}

\begin{theorem}
\label{varianceundir}
Assume that $\beta_{1},\ldots,\beta_{k}$ are all negative and $H_{1}$ denotes a single edge.
Let us further assume that $\lim_{n\rightarrow\infty}n^{2}e^{2\alpha_{n}\beta_{1}}=0$ and \\ $\lim_{n\rightarrow\infty}\frac{\alpha_n}{n}=0$. Then
\begin{equation}
\lim_{n\rightarrow\infty}\frac{\mathbb{P}^{(n)}(X_{1i}=1, X_{1j}=1)}{e^{4\alpha_{n}\beta_1}}=1, \hspace{0.2cm} i\neq j\neq 1.
\end{equation}
\end{theorem}

\begin{proof}
By symmetry, the fact that $X_{1i}, X_{1j}\in\{0,1\}$ and the definition of the probability measure $\mathbb{P}^{(n)}$ \eqref{ll},
\begin{align}
&\mathbb{P}^{(n)}(X_{1i}=1, X_{1j}=1)
=\mathbb{E}^{(n)}[X_{1i}X_{1j}]
\\
&=\frac{1}{(n-1)(n-2)}\left[\mathbb{E}^{(n)}\left[\left(\sum_{i=2}^nX_{1i}\right)^2\right]
-\mathbb{E}^{(n)}\left[\sum_{i=2}^nX_{1i}\right]\right]
\nonumber
\\
&=\frac{1}{(n-1)(n-2)}\frac{2^{\binom{n}{2}}\mathbb{E}\left[\left[\left(\sum_{i=2}^nX_{1i}\right)^2
-\sum_{i=2}^nX_{1i}\right]e^{\alpha_{n}n^{2}\sum_{p=1}^{k}\beta_{p}t(H_{p},G_n)}\right]}
{2^{\binom{n}{2}}\mathbb{E}[e^{\alpha_{n}n^{2}\sum_{p=1}^{k}\beta_{p}t(H_{p},G_n)}]},
\nonumber
\end{align}
where $\mathbb{E}$ is the expectation associated with the uniform measure, i.e., each possible graph configuration is weighted equally at $1/2^{\binom{n}{2}}$.

\noindent We only need to analyze the numerator since the denominator has already been analyzed in Theorem \ref{meanundir}.
On one hand, since $t(H_1, G_n)$ measures the edge density of $G_n$, a graph $G_n$ with only a $2$-star $X_{1i}=X_{1j}=1$ for some $2\leq i\neq j\leq n$ gives $t(H_1, G_n)=4/n^2$ and carries a weight of $1/2^{\binom{n}{2}}$, which implies that
\begin{align}
&2^{\binom{n}{2}}\mathbb{E}\left[\left[\left(\sum_{i=2}^nX_{1i}\right)^2
-\sum_{i=2}^nX_{1i}\right]e^{\alpha_{n}n^{2}\sum_{p=1}^{k}\beta_{p}t(H_{p},G_n)}\right]
\\
&\geq
2^{\binom{n}{2}}\frac{1}{2}\sum_{2\leq i\leq n}\sum_{2\leq j\leq n: i\neq j}
\nonumber
\\
&\mathbb{E}\bigg[(2^2-2)e^{\alpha_{n}n^{2}\sum_{p=1}^{k}\beta_{p}t(H_{p},G_n)}
\bigg|X_{1i}=X_{1j}=1,X_{i'j'}=0,
\nonumber
\\
&\qquad\qquad\qquad\qquad\qquad\qquad\qquad(i',j')\neq(1,i)\text{ and }(i',j')\neq (1,j)\bigg]
\nonumber
\\
&\qquad\qquad\cdot
\mathbb{P}\left( X_{1i}=X_{1j}=1,X_{i'j'}=0,(i',j')\neq(1,i)\text{ and }(i',j')\neq (1,j)\right)
\nonumber
\\
&=e^{4\alpha_{n}\beta_{1}}\sum_{2\leq i\leq n}\sum_{2\leq j\leq n: i\neq j}
\mathbb{E}\bigg[e^{\alpha_{n}n^{2}\sum_{p=2}^{k}\beta_{p}t(H_{p},G_n)}
\bigg|X_{1i}=X_{1j}=1,X_{i'j'}=0,
\nonumber
\\
&\qquad\qquad\qquad\qquad\qquad\qquad\qquad(i',j')\neq(1,i)\text{ and }(i',j')\neq (1,j)\bigg]
\nonumber
\\
&=\sum_{2\leq i\leq n}\sum_{2\leq j\leq n: i\neq j} e^{4\alpha_{n}\beta_{1}+\alpha_n n^2 \sum_{p=2}^k \beta_p c_p n^{-v(H_p)}}
\nonumber
\\
&=(n-1)(n-2)e^{4\alpha_{n}\beta_{1}+\alpha_n n^2 \sum_{p=2}^k \beta_p c_p n^{-v(H_p)}},
\nonumber
\end{align}
where $v(H_p)\geq 3$ denotes the number of vertices of $H_p$ and $c_p$ is the number of homomorphisms from $H_p$ into $G_n$ (a graph with only a $2$-star) and so is a constant that only depends on $H_p$.

\noindent On the other hand, since $\beta_{1},\ldots,\beta_{k}$ are all negative and $t(H_1, G_n)$ measures the edge density of $G_n$, using $(X_{ij})_{1\le i<j \le n}$ are iid Bernoulli under the uniform measure, we have
\begin{align}
&2^{\binom{n}{2}}\mathbb{E}\left[\left[\left(\sum_{i=2}^nX_{1i}\right)^2
-\sum_{i=2}^nX_{1i}\right]e^{\alpha_{n}n^{2}\sum_{p=1}^{k}\beta_{p}t(H_{p},G_n)}\right]
\\
&\leq 2^{\binom{n}{2}}\mathbb{E}\left[\left[\left(\sum_{i=2}^nX_{1i}\right)^2
-\sum_{i=2}^nX_{1i}\right]e^{\alpha_{n}n^{2}\beta_{1}t(H_{1},G_n)}\right]
\nonumber
\\
&=2^{\binom{n}{2}}\mathbb{E}\left[\left[\left(\sum_{i=2}^nX_{1i}\right)^2
-\sum_{i=2}^nX_{1i}\right]e^{2\alpha_{n}\beta_{1}\sum_{1\leq i<j\leq n}X_{ij}}\right]
\nonumber
\\
&=2^{\binom{n}{2}}\left(\frac{1}{4\alpha_{n}^2}\frac{\partial^2}{\partial\beta_{1}^2}
\mathbb{E}\left[e^{2\alpha_{n}\beta_{1}\sum_{i=2}^nX_{1i}}\right]
-\frac{1}{2\alpha_{n}}\frac{\partial}{\partial\beta_{1}}\mathbb{E}\left[e^{2\alpha_{n}\beta_{1}\sum_{i=2}^nX_{1i}}\right]\right)
\nonumber
\\
&\qquad\qquad\qquad\qquad\qquad\qquad\cdot\mathbb{E}\left[e^{2\alpha_n\beta_1\sum_{2\leq i<j\leq n}X_{ij}}\right]
\nonumber
\\
&=2^{\binom{n}{2}}\left(\frac{1}{4\alpha_{n}^2}\frac{\partial^2}{\partial\beta_{1}^2}
\left(\frac{1+e^{2\alpha_{n}\beta_{1}}}{2}\right)^{n-1}-\frac{1}{2\alpha_{n}}
\frac{\partial}{\partial\beta_{1}}\left(\frac{1+e^{2\alpha_{n}\beta_{1}}}{2}\right)^{n-1}\right)
\nonumber
\\
&\qquad\qquad\qquad\qquad\qquad\qquad\qquad\qquad\cdot\left(\frac{1+e^{2\alpha_{n}\beta_{1}}}{2}\right)^{\binom{n-1}{2}}
\nonumber
\\
&=(n-1)(n-2)e^{4\alpha_{n}\beta_{1}}\left(1+e^{2\alpha_{n}\beta_{1}}\right)^{\binom{n}{2}-2}.
\nonumber
\end{align}
The conclusion follows when we apply the scaling assumption.
\end{proof}

Together, Theorems \ref{meanundir} and \ref{varianceundir} indicate that when the rate of divergence of $\alpha_n$ is between the order of $\log n$ and $n$, the graph displays Erd\H{o}s-R\'{e}nyi behavior in the large $n$ limit if all the parameters $\beta_1,\ldots,\beta_k$ are negative, where the edge formation probability is given by $e^{2\alpha_n\beta_1}$. It depends on $\beta_1$ and $n$ and decays to $0$ as $n\rightarrow \infty$. This is in sharp contrast to the standard exponential model where the parameters $\beta_1,\ldots,\beta_k$ are not scaled by $\alpha_n$ and are instead held fixed. In this so-called dense regime, Chatterjee and Diaconis \cite{ChatterjeeDiaconis} have done extensive research and found that when all the parameters $\beta_1,\ldots,\beta_k$ are non-negative, the graph behaves like an Erd\H{o}s-R\'{e}nyi random graph in the large $n$ limit, where the edge formation probability depends on all parameters $\beta_1,\ldots,\beta_k$. Nevertheless, not much is known when some of the parameters are negative. In fact, even when all the parameters are negative as assumed in the present investigation, analysis of the limiting behavior of a typical graph has proved to be very hard, except in some special cases such as the edge-(single)-star and edge-triangle models.

\begin{remark}
We make some comments on the scaling assumption in (\ref{scaling1}) before proceeding further. If a graph is sparse, then the density of edges might not scale with $n$ in the same way as the density of, say, triangles. Intuitively, if the chance of having an edge is small, then the chance of having a triangle would be even smaller. As such, one might expect that $\alpha_n$ should depend on $p$ and diverge to infinity faster when the structure of $H_p$ is more complicated. We generalize (\ref{scaling1}) following this philosophy,
\begin{equation*}
\beta_{p}^{(n)}=\beta_{p}\alpha_{n, p},\quad p=1,2,\ldots,k,
\end{equation*}
and assume that $\lim_{n\rightarrow\infty}n^{2}e^{2\alpha_{n, 1}\beta_{1}}=0$ and $\lim_{n\rightarrow\infty}\frac{\alpha_n, p}{n^{v(H_p)-2}}=0$ for $p=2,\ldots,k$. Under this more relaxed assumption, the proof for Theorems \ref{meanundir} and \ref{varianceundir} go through without modifications and the same conclusions hold.
\end{remark}

As mentioned earlier, we are interested in deriving the exact asymptotics of the limiting log partition function $\log Z_n$ of the exponential model, since it is the generating function for the expectations of all other random variables on the graph space. Due to the myriad of structural possibilities of $H_p$, this is a rather difficult task, but we are able to make some headway in the edge-(single)-star model. We briefly outline the rationale first. For the general exponential model $\mathbb{P}^{(n)}$ defined in (\ref{ll}),
Chatterjee and Dembo \cite{ChatterjeeDembo} showed that when $|\beta_{1}^{(n)}|+\cdots+|\beta_{k}^{(n)}|$
does not grow too fast, $\log Z_{n}/n^2$ may be approximated by
\begin{equation}
\label{add}
\frac{\log Z_n}{n^2} \asymp L_{n}=\sup_{x\in\mathcal{P}_{n}}\left\{\beta_{1}^{(n)}t(H_{1},x)+\cdots+\beta_{k}^{(n)}t(H_{k},x)
-\frac{I(x)}{n^{2}}\right\},
\end{equation}
where $\mathcal{P}_{n}=\{(x_{ij})_{1\leq i<j\leq n}:x_{ij}\in[0,1], 1\leq i<j\leq n\}$,
\begin{equation}
\label{T}
t(H_p, x)=\frac{1}{n^{v(H_p)}}\sum_{q_1,\ldots,q_{v(H_p)}=1}^n\prod_{\{l,l'\}\in E(H_p)}x_{q_l q_{l'}},
\end{equation}
and
\begin{equation*}
I(x)=\sum_{1\leq i<j\leq n}I(x_{ij})=\sum_{1\leq i<j\leq n}\left(x_{ij}\log x_{ij}+(1-x_{ij})\log(1-x_{ij})\right).
\end{equation*}
They also gave a concrete error bound for this approximation,
\begin{align}
\label{couple}
-\frac{cB}{n}
&\leq\frac{\log Z_{n}}{n^{2}}-L_{n}
\\
&\leq
CB^{8/5}n^{-1/5}(\log n)^{1/5}\left(1+\frac{\log B}{\log n}\right)+CB^{2}n^{-1/2},
\nonumber
\end{align}
where $B=|\beta_{1}^{(n)}|+\cdots+|\beta_{k}^{(n)}|$, and $c$ and $C$ are constants that may depend only on $H_{1},\ldots,H_{k}$.

For the edge-$p$-star model, i.e., $H_1$ is an edge and $H_2$ is a $p$-star, (\ref{add}) becomes
\begin{equation}
\label{OP}
L_n=\sup_{x\in\mathcal{P}_{n}}\left\{\alpha_n\beta_{1}t(H_{1},x)+\alpha_n\beta_{2}t(H_{2},x)
-\frac{I(x)}{n^{2}}\right\}.
\end{equation}
The supremum over upper triangular array $\mathcal{P}_n$ in (\ref{OP}) may be simplified further. On one hand, it was proved in Chatterjee and Diaconis \cite{ChatterjeeDiaconis} that when $H_{2}$ is a $p$-star,
\begin{equation}
L_n\leq \sup_{0\leq x^*\leq 1}\left\{\alpha_{n}\beta_{1}x^*+\alpha_{n}\beta_{2}(x^*)^{p}-\frac{1}{2}I(x^*)\right\}.
\end{equation}
On the other hand, by considering $(x_{ij}) \in \mathcal{P}_n$ where $x_{ij}\equiv x^*$ for any $1\leq i<j\leq n$,
\begin{equation}
L_{n}\geq \left(1-\frac{c}{n}\right)\sup_{0\leq x^*\leq 1}\left\{\alpha_{n}\beta_{1}x^*+\alpha_{n}\beta_{2}(x^*)^{p}-\frac{1}{2}I(x^*)\right\},
\end{equation}
where $c$ is a constant that only depends on $H_1$ and $H_2$. This extra dependency comes from the following consideration. The difference between $I(x)/n^2$
and $I(x^*)/2$ is easy to estimate, while the difference between $t(H_1,x)$ and $x^*$ (or between $t(H_2,x)$ and $(x^*)^{p}$) is
caused by the zero diagonal terms $x_{ii}$. We do a broad estimate and find that it is bounded by $c_p/n$, where $c_p$ is a constant that only depends on $H_p$.
Therefore the upper and lower bounds for $L_n$ are asymptotically the same,
\begin{equation}\label{Approx}
L_{n} \asymp \sup_{0\leq x^*\leq 1}\left\{\alpha_{n}\beta_{1}x^*+\alpha_{n}\beta_{2}(x^*)^{p}-\frac{1}{2}I(x^*)\right\}.
\end{equation}
The following Theorem \ref{sparse} explores the asymptotics for $L_n$ using (\ref{Approx}) and in turn provides an exact asymptotic estimate for $\log Z_n$ based on (\ref{couple}).

\begin{theorem}
\label{sparse}
Consider the edge-$p$-star model,  i.e., $H_{1}$ is an edge and $H_{2}$ is a $p$-star. Assume that $\beta_1$ and $\beta_2$ are both negative. Then
\begin{equation}
\lim_{n\rightarrow\infty}\frac{L_{n}}{e^{2\alpha_{n}\beta_{1}}}=\frac{1}{2}.
\end{equation}
Let us further assume that $\lim_{n\rightarrow\infty}\frac{\alpha_{n}^{8/5}(\log n)^{1/5}e^{2\alpha_{n}|\beta_{1}|}}{n^{1/5}}=0$. Then
\begin{equation}
\lim_{n\rightarrow\infty}\frac{\log Z_{n}}{n^{2}e^{2\alpha_{n}\beta_{1}}}=\frac{1}{2}.
\end{equation}
\end{theorem}

\begin{proof}
The optimization problem \eqref{Approx} was well studied in Radin and Yin \cite{Radin} and Aristoff and Zhu \cite{AristoffZhu}. When $\beta_1$ and $\beta_2$ are both negative, they showed that the optimizer $x^*$ uniquely satisfies
\begin{equation}\label{FirstDerivative}
\alpha_{n}\beta_{1}+\alpha_{n}\beta_{2}p(x^{\ast})^{p-1}=\frac{1}{2}\log\left(\frac{x^{\ast}}{1-x^{\ast}}\right).
\end{equation}
Since $|\alpha_{n}\beta_{1}|$ and $|\alpha_{n}\beta_{2}|$ both diverge to infinity as $n\rightarrow\infty$, $x^{\ast}\rightarrow 0$ as $n\rightarrow\infty$. We can rewrite \eqref{FirstDerivative} as
\begin{equation}\label{FirstRe}
(1-x^{\ast})e^{2\alpha_{n}\beta_{2}p(x^{\ast})^{p-1}}=\frac{x^{\ast}}{e^{2\alpha_{n}\beta_{1}}}.
\end{equation}
This shows that $\frac{x^{\ast}}{e^{2\alpha_{n}\beta_{1}}}\leq 1$,
and thus
\begin{equation}
2\alpha_{n}\beta_{2}p(x^{\ast})^{p-1}=2\alpha_{n}\beta_{2}pe^{2(p-1)\alpha_{n}\beta_{1}}\left(\frac{x^{\ast}}{e^{2\alpha_{n}\beta_{1}}}\right)^{p-1}
\rightarrow 0
\end{equation}
as $n\rightarrow\infty$. Hence, we conclude that
\begin{equation}
\lim_{n\rightarrow\infty}\frac{x^{\ast}}{e^{2\alpha_{n}\beta_{1}}}=1.
\end{equation}

By (\ref{Approx}) and \eqref{FirstDerivative},
\begin{align}
L_{n}&=
\alpha_{n}\beta_{1}x^{\ast}+\alpha_{n}\beta_{2}(x^{\ast})^{p}-\frac{1}{2}x^{\ast}\log x^{\ast}-\frac{1}{2}(1-x^{\ast})\log(1-x^{\ast})
\\
&=\alpha_{n}\beta_{2}(x^{\ast})^{p}(1-p)-\frac{1}{2}\log(1-x^{\ast})
\nonumber
\\
&=\alpha_{n}\beta_{2}(x^{\ast})^{p}(1-p)+\frac{1}{2}x^{\ast}+O((x^{\ast})^{2})
\nonumber,
\end{align}
which implies that $\lim_{n\rightarrow\infty}\frac{L_{n}}{x^{\ast}}=\frac{1}{2}$ and
$\lim_{n\rightarrow\infty}\frac{L_{n}}{e^{2\alpha_{n}\beta_{1}}}=\frac{1}{2}$. From (\ref{couple}),
\begin{align}
&-\frac{c\alpha_{n}(|\beta_{1}|+|\beta_{2}|)}{n} \leq\frac{\log Z_{n}}{n^{2}}-L_{n}
\\
&\leq
C\alpha_{n}^{8/5}(|\beta_{1}|+|\beta_{2}|)^{8/5}
n^{-1/5}(\log n)^{1/5}\left(1+\frac{\log\alpha_{n}+\log(|\beta_{1}|+|\beta_{2}|)}{\log n}\right)
\nonumber
\\
&\qquad\qquad
+C(|\beta_{1}|+|\beta_{2}|)^{2}\alpha_{n}^{2}n^{-1/2}.
\nonumber
\end{align}
Therefore, under the further assumption that
$\lim_{n\rightarrow\infty}\frac{\alpha_{n}^{8/5}(\log n)^{1/5}e^{2\alpha_{n}|\beta_{1}|}}{n^{1/5}}=0$,
we have
\begin{equation}
\left|\frac{\log Z_{n}}{n^{2}e^{2\alpha_{n}\beta_{1}}}-\frac{L_{n}}{e^{2\alpha_{n}\beta_{1}}}\right|
\rightarrow 0\qquad\text{as $n\rightarrow\infty$},
\end{equation}
and hence
\begin{equation}
\lim_{n\rightarrow\infty}\frac{\log Z_{n}}{n^{2}e^{2\alpha_{n}\beta_{1}}}=\frac{1}{2}.
\end{equation}
\end{proof}

\begin{remark}
This is consistent with the observations in Theorems \ref{meanundir} and \ref{varianceundir}, since the log partition function $\log Z_n$ of an Erd\H{o}s-R\'{e}nyi random graph with edge formation probability $e^{2\alpha_n\beta_1}$ may be approximated by
\begin{equation}
\frac{\log Z_n}{n^2} \asymp -\frac{1}{2}\log(1-e^{2\alpha_n\beta_1}) \asymp \frac{1}{2}e^{2\alpha_n\beta_1}.
\end{equation}
\end{remark}

\section{Directed Graphs}
\label{directed}
Consider directed graphs $G_n$ on $n$ vertices,
where a graph is represented by a matrix
$X_n=(X_{ij})_{1\le i,j \le n}$ with each $X_{ij} \in \{0,1\}$.
Here, $X_{ij} = 1$ means there is a directed edge from vertex $i$ to
vertex $j$; otherwise, $X_{ij}=0$. Give the set of such graphs the probability
\begin{equation}\label{1}
{\mathbb P}^{(n)}(G_n) = Z_n(\beta_{1}^{(n)},\beta_{2}^{(n)},\ldots,\beta_{k}^{(n)})^{-1}\exp\left[n^2\left(\sum_{p=1}^{k}\beta_{p}^{(n)} s_{p}(X_n)\right)\right],
\end{equation}
where
\begin{equation}\label{es}
s_{p}(X_n)= n^{-p-1} \sum_{1\leq i,j_{1},j_{2},\ldots,j_{p}\leq n}X_{ij_{1}}X_{ij_{2}}\cdots X_{ij_{p}}
\end{equation}
is the directed $p$-star homomorphism density of $G_n$ and $Z_n(\beta_{1}^{(n)},\beta_{2}^{(n)},\ldots,\beta_{k}^{(n)})$ is the appropriate normalization. The parameters $\beta_p^{(n)}$ are
scaled according to the number of vertices of the graph,
\begin{equation}
\label{scaling2}
\beta_{p}^{(n)}=\beta_{p}\alpha_{n},\quad p=1,2,\ldots,k,
\end{equation}
where $\alpha_{n}\rightarrow\infty$ as $n\rightarrow\infty$. Note that $s_{p}(X_n)$ has an alternate expression
\begin{equation}
\label{alternate}
s_{p}(X_n) = n^{-p-1} \sum_{i=1}^n \left(\sum_{j=1}^n X_{ij}\right)^p,
\end{equation}
and in particular, when $p=1$, it represents the directed edge homomorphism density of $G_n$.
For ease of notation, we allow $X_{ii}$ to equal $1$, but it is not hard to see that without this simplification,
our main results still hold. As in the undirected case, we are interested in the situation where a typical random graph
sampled from the exponential model $\mathbb P^{(n)}$ (\ref{1}) is sparse, i.e., the probability that there is a directed
edge from vertex $i$ to vertex $j$ goes to $0$ as $n\rightarrow \infty$. We ask the same question: When the parameters $\beta_i$ are all negative, what is the speed of the graph towards sparsity? The following Remark \ref{remark} and Theorems \ref{meandir} and \ref{variancedir} explore the asymptotics of the limiting log partition function $\log Z_n$ and the mean and variance of the limiting probability distribution $\mathbb P^{(n)}$ and provide some concrete answers under a sparsity assumption about the parameters. Together, they show that when the rate of divergence of $\alpha_n$ is between the order of $\log n$ and $n$, the graph displays Erd\H{o}s-R\'{e}nyi behavior in the large $n$ limit, where the edge formation probability is given by $e^{\alpha_n\beta_1}$. It depends on $\beta_1$ and $n$ and decays to $0$ as $n\rightarrow \infty$.

\begin{remark}
\label{remark}
(i) When $\beta_{1},\ldots,\beta_{k}$ are all negative, on one hand,
\begin{equation}
(Z_{n})^{\frac{1}{n}}\leq\sum_{j=0}^{n}\binom{n}{j}e^{n\alpha_{n}\beta_{1}(\frac{j}{n})}
=(1+e^{\alpha_{n}\beta_{1}})^{n},
\end{equation}
which implies that $\limsup_{n\rightarrow\infty}(Z_{n})^{\frac{1}{n^{2}}}\leq 1$. On the other hand,
\begin{equation}
(Z_{n})^{\frac{1}{n}}\geq\binom{n}{0}e^{n\alpha_{n}\sum_{p=1}^{k}\beta_{p}(\frac{0}{n})^{p}}=1.
\end{equation}
Therefore
\begin{equation}
\lim_{n\rightarrow\infty}(Z_{n})^{\frac{1}{n^{2}}}=1.
\end{equation}

(ii) Furthermore, if we assume that $\lim_{n\rightarrow\infty}ne^{\alpha_{n}\beta_{1}}=0$, then
\begin{equation}
\lim_{n\rightarrow\infty}(Z_{n})^{\frac{1}{n}}=1.
\end{equation}

(iii) If instead we assume that $\lim_{n\rightarrow \infty}ne^{\alpha_{n}\beta_{1}}=\lambda \in (0, \infty)$, then as will be shown in Theorem \ref{caselambda},
\begin{equation}
\lim_{n\rightarrow\infty}(Z_n)^{\frac{1}{n}}=e^{\lambda}.
\end{equation}

(iv) We can get more precise asymptotics.
Let us assume that \\ $\lim_{n\rightarrow\infty}ne^{\alpha_{n}\beta_{1}}=0$ and $\lim_{n\rightarrow\infty}\frac{\alpha_{n}}{n}=0$.
On one hand,
\begin{equation}
\frac{\log Z_{n}}{n^{2}}\leq\log(1+e^{\alpha_{n}\beta_{1}}),
\end{equation}
which implies that
\begin{equation}
\limsup_{n\rightarrow\infty}\frac{\log Z_{n}}{n^{2}e^{\alpha_n\beta_1}}\leq 1.
\end{equation}
On the other hand,
\begin{equation}
(Z_{n})^{\frac{1}{n}}\geq\sum_{j=0}^{1}\binom{n}{j}e^{n\alpha_{n}\sum_{p=1}^{k}\beta_{p}(\frac{j}{n})^{p}}
=1+ne^{n\alpha_{n}\sum_{p=1}^{k}\beta_{p}(\frac{1}{n})^{p}},
\end{equation}
which implies that
\begin{align}
\liminf_{n\rightarrow\infty}\frac{\log Z_{n}}{n^{2}e^{\alpha_{n}\beta_{1}}}
&\geq\liminf_{n\rightarrow\infty}\frac{\log(1+ne^{n\alpha_{n}\sum_{p=1}^{k}\beta_{p}(\frac{1}{n})^{p}})}{ne^{\alpha_{n}\beta_{1}}}
\\
&=\liminf_{n\rightarrow\infty}\frac{\log(1+ne^{n\alpha_{n}\sum_{p=1}^{k}\beta_{p}(\frac{1}{n})^{p}})}{ne^{n\alpha_{n}\sum_{p=1}^{k}\beta_{p}(\frac{1}{n})^{p}}}
\frac{ne^{n\alpha_{n}\sum_{p=1}^{k}\beta_{p}(\frac{1}{n})^{p}}}{ne^{\alpha_{n}\beta_{1}}}
\nonumber
\\
&=1.\nonumber
\end{align}
Therefore
\begin{equation}
\lim_{n\rightarrow\infty}\frac{\log Z_{n}}{n^{2}e^{\alpha_{n}\beta_{1}}}=1.
\end{equation}
\end{remark}

\begin{theorem}
\label{meandir}
Assume that $\beta_{1},\ldots,\beta_{k}$ are all negative. Let us further assume that $\lim_{n\rightarrow\infty}ne^{\alpha_{n}\beta_{1}}=0$
and $\lim_{n\rightarrow\infty}\frac{\alpha_{n}}{n}=0$. Then
\begin{equation}
\lim_{n\rightarrow\infty}\frac{\mathbb{P}^{(n)}(X_{1i}=1)}{e^{\alpha_{n}\beta_{1}}}=1.
\end{equation}
\end{theorem}

\begin{proof}
The proof follows a similar line of reasoning as in the proof of Theorem \ref{meanundir}. However, due to the alternate expression of the directed star density (\ref{alternate}), rather than concentrating on each single edge, we will examine the number of directed edges from vertex $1$ as a whole, which is Binomial under the uniform measure. Using symmetry, we write the probability of a directed edge as a quotient of two expectations. The lower bound for the denominator is obtained by considering a graph with no edges protruding from vertex $1$, and the lower bound for the numerator is obtained by considering a graph with only one edge protruding from vertex $1$. The upper bound for both the numerator and the denominator is obtained by including only the directed edge density in the exponent.
\end{proof}

\begin{theorem}
\label{variancedir}
Assume that $\beta_{1},\ldots,\beta_{k}$ are all negative. Let us further assume that $\lim_{n\rightarrow\infty}ne^{\alpha_{n}\beta_{1}}=0$
and $\lim_{n\rightarrow\infty}\frac{\alpha_{n}}{n}=0$. Then
\begin{equation}
\lim_{n\rightarrow\infty}\frac{\mathbb{P}^{(n)}(X_{1i}=1,X_{1j}=1)}
{e^{2\alpha_{n}\beta_{1}}}=1, \hspace{0.2cm} i\neq j.
\end{equation}
\end{theorem}

\begin{proof}
The proof follows a similar line of reasoning as in the proof of Theorem \ref{varianceundir}. Using symmetry, we write the probability of a directed $2$-star as a quotient of two expectations. The lower bound for the denominator is obtained by considering a graph with no edges protruding from vertex $1$, and the lower bound for the numerator is obtained by considering a graph with two edges protruding from vertex $1$. The upper bound for both the numerator and the denominator is obtained by including only the directed edge density in the exponent.
\end{proof}

\begin{remark}
As in the undirected case, we remark that the scaling assumption in (\ref{scaling2}) may be generalized by taking
\begin{equation*}
\beta_{p}^{(n)}=\beta_{p}\alpha_{n, p},\quad p=1,2,\ldots,k,
\end{equation*}
and assuming that $\lim_{n\rightarrow\infty}ne^{\alpha_{n, 1}\beta_{1}}=0$ and $\lim_{n\rightarrow\infty}\frac{\alpha_n, p}{n^{p-1}}=0$ for $p=2,\ldots,k$. Under this more relaxed assumption, the proof for Theorems \ref{meandir} and \ref{variancedir} go through without modifications and the same conclusions hold.
\end{remark}

We now ask some related questions. Consider an Erd\H{o}s-R\'{e}nyi random graph on $n$ vertices with edge formation probability $\rho$. The distribution of the degree of any vertex $i$ is Binomial with parameters $n$ and $\rho$. A known fact is that for $n$ large, $\rho$ small and $n\rho$ a constant, Binomial distribution with these parameters tends to a Poisson distribution with parameter $n\rho$. We have seen in Theorems \ref{meandir} and \ref{variancedir} that when the rate of divergence of $\alpha_n$ is between the order of $\log n$ and $n$, the graph displays Erd\H{o}s-R\'{e}nyi behavior, where the edge formation probability is given by $e^{\alpha_n\beta_1}$. One natural question to ask is if $ne^{\alpha_n\beta_1}$ approaches a constant $\lambda \in (0,\infty)$ as $n\rightarrow \infty$, i.e., when the divergence rate of $\alpha_n$ is of the order of $\log n$, will the graph display Poisson behavior? The following Theorem \ref{caselambda} gives an affirmative answer to this question. Notice that if $\lim_{n\rightarrow\infty}ne^{\alpha_{n}\beta_{1}}=\lambda$, then $\lim_{n\rightarrow\infty}\frac{\alpha_{n}}{n}=0$ is automatically satisfied.

\begin{theorem}
\label{caselambda}
Assume that $\beta_{1},\ldots,\beta_{k}$ are all negative. Let us further assume that
$\lim_{n\rightarrow\infty}ne^{\alpha_{n}\beta_{1}}=\lambda\in(0,\infty)$.
Then
\begin{equation}
\lim_{n\rightarrow\infty}\frac{\mathbb{P}^{(n)}(X_{1i}=1)}{\lambda n^{-1}}=1.
\end{equation}
\begin{equation}
\lim_{n\rightarrow\infty}\frac{\mathbb{P}^{(n)}(X_{1i}=1,X_{1j}=1)}{\lambda^2 n^{-2}}=1, \hspace{0.2cm} i\neq j.
\end{equation}
Moreover, the degree of any vertex is asymptotically Poisson with parameter $\lambda$, i.e.,
\begin{equation}
\sum_{i=1}^{n}X_{1i}\rightarrow\text{Poisson}(\lambda)
\end{equation}
in distribution as $n\rightarrow\infty$.
\end{theorem}

\begin{proof}
By symmetry,
\begin{align}
\mathbb{P}^{(n)}(X_{1i}=1)&=\frac{1}{n}\mathbb{E}^{(n)}\left[\sum_{i=1}^{n}X_{1i}\right]
\\
&=\frac{1}{n}\frac{\sum_{j=0}^{n}\binom{n}{j}je^{\alpha_{n}\beta_{1}j+\sum_{p=2}^{k}\alpha_{n}\beta_{p}\frac{j^{p}}{n^{p-1}}}}
{\sum_{j=0}^{n}\binom{n}{j}e^{\alpha_{n}\beta_{1}j+\sum_{p=2}^{k}\alpha_n\beta_{p}\frac{j^{p}}{n^{p-1}}}}.
\nonumber
\end{align}
First, let us analyze the denominator. On one hand, for any fixed $M$,
\begin{align}
\sum_{j=0}^{n}\binom{n}{j}e^{\alpha_{n}\beta_{1}j+\sum_{p=2}^{k}\alpha_n\beta_{p}\frac{j^{p}}{n^{p-1}}}
&\geq \sum_{j=0}^M \binom {n}{j} e^{\alpha_n\beta_1 j+\sum_{p=2}^k \alpha_n\beta_p \frac{j^p}{n^{p-1}}}
\\
&\rightarrow \sum_{j=0}^M \frac{\lambda^j}{j!}.
\nonumber
\end{align}
Since this is true for any $M$, let $M \rightarrow \infty$, and we obtain an asymptotic
lower bound $\sum_{j=0}^{\infty}\frac{\lambda^j}{j!}=e^{\lambda}$. On the other hand,
\begin{align}
\sum_{j=0}^{n}\binom{n}{j}e^{\alpha_{n}\beta_{1}j+\sum_{p=2}^{k}\alpha_{n}\beta_{p}\frac{j^{p}}{n^{p-1}}}
&\leq \sum_{j=0}^{n}\binom{n}{j}e^{\alpha_{n}\beta_{1}j}
\\
&=(1+e^{\alpha_n\beta_1})^n
\nonumber
\\
&\rightarrow e^{\lambda}.
\nonumber
\end{align}
Next, let us analyze the numerator. On one hand, for any fixed $M$,
\begin{align}
\sum_{j=0}^{n}j\binom{n}{j}e^{\alpha_{n}\beta_{1}j+\sum_{p=2}^{k}\alpha_{n}\beta_{p}\frac{j^{p}}{n^{p-1}}}
&\geq \sum_{j=0}^M j\binom{n}{j}e^{\alpha_{n}\beta_{1}j+\sum_{p=2}^{k}\alpha_{n}\beta_{p}\frac{j^{p}}{n^{p-1}}}
\\
&\rightarrow \sum_{j=0}^{M-1} \frac{\lambda^{j+1}}{j!}.
\nonumber
\end{align}
Since this is true for any $M$, let $M \rightarrow \infty$, and we obtain an asymptotic
lower bound $\sum_{j=0}^{\infty} \frac{\lambda^{j+1}}{j!}=\lambda e^{\lambda}$. On the other hand,
\begin{align}
\sum_{j=0}^{n}j\binom{n}{j}e^{\alpha_{n}\beta_{1}j+\sum_{p=2}^{k}\alpha_{n}\beta_{p}\frac{j^{p}}{n^{p-1}}}
&\leq \sum_{j=0}^{n}j\binom{n}{j}e^{\alpha_{n}\beta_{1}j}
\\
&=\frac{1}{\alpha_n}\frac{\partial}{\partial \beta_1}(1+e^{\alpha_n\beta_1})^n
\nonumber
\\
&=ne^{\alpha_n\beta_1}(1+e^{\alpha_n\beta_1})^{n-1}
\nonumber
\\
&\rightarrow \lambda e^{\lambda}.
\nonumber
\end{align}

Again by symmetry,
\begin{align}
\mathbb{P}^{(n)}(X_{1i}=1,X_{1j}=1)
&=\frac{1}{n(n-1)}\left[\mathbb{E}^{(n)}\left[\left(\sum_{i=1}^{n}X_{1i}\right)^{2}\right]-\mathbb{E}^{(n)}\left[\sum_{i=1}^{n}X_{1i}\right]\right]
\\
&=\frac{1}{n(n-1)}\frac{\sum_{j=0}^{n}(j^{2}-j)\binom{n}{j}e^{\alpha_{n}\beta_{1}j+\sum_{p=2}^{k}\alpha_n\beta_{p}\frac{j^{p}}{n^{p-1}}}}
{\sum_{j=0}^{n}\binom{n}{j}e^{\alpha_{n}\beta_{1}j+\sum_{p=2}^{k}\alpha_n\beta_{p}\frac{j^{p}}{n^{p-1}}}}.
\nonumber
\end{align}
We have observed earlier that the denominator converges to $e^{\lambda}$ as $n\rightarrow \infty$. Let us analyze the numerator. On one hand, for any fixed $M$,
\begin{align}
\sum_{j=0}^{n}(j^{2}-j)\binom{n}{j}e^{\alpha_{n}\beta_{1}j+\sum_{p=2}^{k}\alpha_n\beta_{p}\frac{j^{p}}{n^{p-1}}}
&\geq \sum_{j=0}^{M}(j^{2}-j)\binom{n}{j}e^{\alpha_{n}\beta_{1}j+\sum_{p=2}^{k}\alpha_n\beta_{p}\frac{j^{p}}{n^{p-1}}}
\\
&\rightarrow \sum_{j=0}^{M-2}\frac{\lambda^{j+2}}{j!}.
\nonumber
\end{align}
Since this is true for any $M$, let $M\rightarrow \infty$, and we obtain an asymptotic lower bound $\sum_{j=0}^{\infty}\frac{\lambda^{j+2}}{j!}=\lambda^2 e^{\lambda}$.
On the other hand,
\begin{align}
&\sum_{j=0}^{n}(j^{2}-j)\binom{n}{j}e^{\alpha_{n}\beta_{1}j+\sum_{p=2}^{k}\alpha_n\beta_{p}\frac{j^{p}}{n^{p-1}}}
\\
&\leq \sum_{j=0}^n (j^2-j)\binom{n}{j} e^{\alpha_n\beta_1j}
\\
\nonumber
&=\frac{1}{\alpha_n^2}\frac{\partial^2}{\partial \beta_1^2}(1+e^{\alpha_n\beta_1})^n-\frac{1}{\alpha_n}\frac{\partial}{\partial \beta_1}(1+e^{\alpha_n\beta_1})^n
\nonumber
\\
&=n(n-1)e^{2\alpha_n\beta_1}(1+e^{\alpha_n\beta_1})^{n-2} \rightarrow \lambda^2 e^{\lambda}.
\nonumber
\end{align}

Lastly, for any fixed $j\in\mathbb{N}\cup\{0\}$,
\begin{align}
\mathbb{P}^{(n)}\left(\sum_{i=1}^{n}X_{1i}=j\right)
&=\mathbb{E}^{(n)}\left[1_{\sum_{i=1}^{n}X_{1i}=j}\right]
\\
&=\frac{\binom{n}{j}e^{\alpha_{n}\beta_{1}j+\sum_{p=2}^{k}\alpha_{n}\beta_{p}\frac{j^{p}}{n^{p-1}}}}
{\sum_{j=0}^{n}\binom{n}{j}e^{\alpha_{n}\beta_{1}j+\sum_{p=2}^{k}\alpha_n\beta_{p}\frac{j^{p}}{n^{p-1}}}},
\nonumber
\end{align}
and the denominator converges to $e^{\lambda}$ and the numerator converges to $\frac{\lambda^j}{j!}$ as $n\rightarrow \infty$.

\end{proof}

A natural follow-up question is what if the divergence rate of $\alpha_n$ is of the order of $n$ or faster? Since dependence on the rest of the parameters $\beta_2,\ldots,\beta_k$ does not diminish as $n\rightarrow \infty$, this situation is much harder to study. Some partial answers are given in the following Theorem \ref{infdir}. Notice that if $\liminf_{n\rightarrow\infty}\frac{\alpha_{n}}{n}>0$, then $\lim_{n\rightarrow\infty}ne^{\alpha_{n}\beta_{1}}=0$ is automatically satisfied.

\begin{theorem}
\label{infdir}
Assume that $\beta_{1},\ldots,\beta_{k}$ are all negative. Let us further assume that $\liminf_{n\rightarrow\infty}\frac{\alpha_{n}}{n}>\frac{\log 2}{|\beta_{1}|}$. Then
\begin{equation}
\lim_{n\rightarrow\infty}\frac{\mathbb{P}^{(n)}(X_{1i}=1)}{e^{n\alpha_{n}\sum_{p=1}^{k}\beta_{p}(\frac{1}{n})^{p}}}=1.
\end{equation}
\end{theorem}

\begin{proof}
By symmetry,
\begin{equation}
\mathbb{P}^{(n)}(X_{1i}=1)
=\frac{1}{n}\frac{\sum_{j=0}^{n}\binom{n}{j}je^{n\alpha_{n}\sum_{p=1}^{k}\beta_{p}(\frac{j}{n})^{p}}}
{\sum_{j=0}^{n}\binom{n}{j}e^{n\alpha_{n}\sum_{p=1}^{k}\beta_{p}(\frac{j}{n})^{p}}}.
\end{equation}
From Remark \ref{remark} (ii), the denominator converges to $1$ as $n\rightarrow \infty$. Let us analyze the numerator. On one hand,
\begin{equation}
\sum_{j=0}^{n}\binom{n}{j}je^{n\alpha_{n}\sum_{p=1}^{k}\beta_{p}(\frac{j}{n})^{p}}
\geq ne^{n\alpha_{n}\sum_{p=1}^{k}\beta_{p}(\frac{1}{n})^{p}}.
\end{equation}
On the other hand,
\begin{equation}
\frac{2^{n}ne^{2\alpha_{n}\beta_{1}}}
{ne^{n\alpha_{n}\sum_{p=1}^{k}\beta_{p}(\frac{1}{n})^{p}}}
=\frac{2^{n}e^{2\alpha_{n}\beta_{1}}}
{e^{\alpha_{n}\beta_{1}+\sum_{p=2}^{k}\alpha_{n}\frac{\beta_{p}}{n^{p-1}}}}
=e^{n[\log 2+\frac{\alpha_{n}}{n}\beta_{1}-\frac{\alpha_{n}}{n}\sum_{p=2}^{k}\frac{\beta_{p}}{n^{p-1}}]}\rightarrow 0
\end{equation}
implies that
\begin{align}
&\sum_{j=0}^{n}\binom{n}{j}je^{n\alpha_{n}\sum_{p=1}^{k}\beta_{p}(\frac{j}{n})^{p}}
\leq ne^{n\alpha_{n}\sum_{p=1}^{k}\beta_{p}(\frac{1}{n})^{p}}+\sum_{j=2}^{n}\binom{n}{j}je^{n\alpha_{n}\beta_{1}\frac{2}{n}}
\\
&\leq ne^{n\alpha_{n}\sum_{p=1}^{k}\beta_{p}(\frac{1}{n})^{p}}+ 2^{n}ne^{2\alpha_{n}\beta_{1}}\asymp ne^{n\alpha_{n}\sum_{p=1}^{k}\beta_{p}(\frac{1}{n})^{p}}.
\nonumber
\end{align}
\end{proof}

\section*{Acknowledgements}
The authors are very grateful to the anonymous referee for the invaluable suggestions that greatly improved the quality of this paper.
Mei Yin's research was partially supported by NSF grant DMS-1308333. She appreciated the opportunity to talk about this work in the 2015 ICERM workshop on Crystals, Quasicrystals and Random Networks, organized by Mark Bowick, Persi Diaconis, Charles Radin, and Peter Winkler. She thanks Lorenzo Sadun for his kind and helpful suggestions and comments.

\end{document}